\documentclass[a4paper, 12pt]{article}
\usepackage{}

\usepackage{mathrsfs}
\usepackage{epsfig}
\usepackage{amsmath}
\usepackage{amssymb}
\usepackage{latexsym}
\usepackage{amsfonts}
\usepackage{amsthm}

\usepackage{bbm}

\usepackage[numbers,sort&compress]{natbib}
\usepackage{graphicx}
\ifx\pdfoutput\undefined  \else
 \fi

\usepackage[centerlast]{caption2}

\usepackage{color}

\usepackage{txfonts}
\usepackage{amssymb}
\usepackage{mathrsfs}
\usepackage{amsfonts}

\newtheorem{theorem}{Theorem}[section]
\newtheorem{lemma}[theorem]{Lemma}

\newtheorem{prop}[theorem]{Proposition}

\newtheorem{remark}[theorem]{Remark}
\newtheorem{definition}[theorem]{Definition}
\newtheorem{exam}[theorem]{Example}
\newtheorem{problem}[theorem]{Problem}
\newtheorem{ques}[theorem]{Question}

\usepackage{latexsym,amssymb}
\usepackage{amsmath}

\begin{document}

\title{{Enumeration of idempotent-sum subsequences in finite cyclic semigroups and smooth sequences}}
\author{
Guoqing Wang\footnote{Corresponding Email: gqwang1979@aliyun.com}    \ \ \ \ \ \ \ Yang Zhao \ \ \ \ \ \ \ Xingliang Yi \\
\small{School of
Mathematical Sciences, Tiangong University, Tianjin, 300387, P. R. China}\\
}
\date{}
\maketitle

\begin{abstract} The enumeration of zero-sum subsequences of a given sequence over finite cyclic groups is one classical topic, which starts from one question of P. Erd\H{o}s. In this paper, we consider this problem in a more general setting---finite cyclic semigroups. Let $\mathcal{S}$ be a finite cyclic semigroup.
By $\textbf{e}$ we denote the unique idempotent of the semigroup $\mathcal{S}$. Let $T$ be a sequence over the semigroup $\mathcal{S}$, and let
$N(T; \textbf{e})$ be the number of distinct subsequences of $T$ with sum being the idempotent $\textbf{e}$. We obtain the lower bound for $N(T; \textbf{e})$ in terms of the length of $T$, and moreover, prove that $T$ contains subsequences with some smooth-structure in case that $N(T; \textbf{e})$ is not large. Our result generalizes the theorem obtained by W. Gao [Discrete Math., 1994] on the enumeration of zero-sum subsequences over finite cyclic groups to the setting of semigroups.
\end{abstract}

\noindent{\small {\bf Key Words}: {\sl  Idempotent-sum sequences; Enumeration of zero-sum subsequences; Smooth sequences; Signed smooth sequences;  Erd\H{o}s-Burgess constant; Davenport constant; Zero-sum;  Cyclic semigroups}}

\section {Introduction}

Let $G$ be a finite abelian group written additively. Let $T=(a_1,\ldots,a_{\ell})$ be a sequence over $G$.
By $N(T; 0_G)$ (just $N(T)$ for short) we denote the number of subsequences of $T$ with sum being the zero $0_G$ of $G$, i.e., the number of solutions $(\epsilon_1, \epsilon_2,\ldots,\epsilon_{\ell})$ of the form $\epsilon_i\in \{0,1\}$ for the equation $$\epsilon_1 a_1+\epsilon_2 a_2+\cdots+ \epsilon_{\ell} a_{\ell}=0_G.$$
The enumeration of subsequences with some prescribed properties, problems of this flavour as $N(T; 0_G)$,  is
a classical topic in Combinatorial Number Theory and starting from the following question asked by P. Erd\H{o}s  (personal letter from R.L. Graham as reported in \cite{OlsonEnumberate}):

{\sl ``If $T=(a_l, a_2,\ldots,a_p)$ is a sequence over the cyclic group $\mathbb{Z}/p\mathbb{Z}$ ($p$ is prime) such that $a_i$ are not all equal and $a_i\neq 0$ ($i\in \{1,2,\ldots,p\}$),
is it true that $N(T)\geq p$?''}

The Erd\H{o}s' question was completely answered by J.E. Olson who proved the following:

\noindent \textbf{Theorem A.} \cite{OlsonEnumberate} \ {\sl Let $n>0$ be an integer, and let $T=(a_1, a_2,\ldots,a_n)$ be a sequence over the cyclic group $\mathbb{Z}/n\mathbb{Z}$ such that $a_i$ are not all equal and $a_i\neq 0$ ($i\in \{1,2,\ldots,n\}$). Then $N(T)\geq n$.}

In the meantime, variety of results on the enumeration of subsequences over finite abelian groups were obtained (see e.g., \cite{CaoSun,uredi,Gaonumbertheory,ChangQuWang,GaoDisenumber,Gaogivenenumber,GaoGeroldingernumber1,GGXia,GPeng,Grynkiewnumber,Kisin}). Among which, W. Gao \cite{GaoDisenumber} obtained the following theorem  which characterized all sequences $T$ over the cyclic group $\mathbb{Z}/n\mathbb{Z}$ of order $n$ such that $N(T)$ is not large.

\noindent \textbf{Theorem B.} \cite{GaoDisenumber} \ {\sl Let $n>0$, and let $T$ be a sequence over $\mathbb{Z}/n\mathbb{Z}$. For any integer $\delta$ such that $1\leq \delta\leq n/4+1$,
if $N(T)<2^{|T|-n+1+\delta}$ then there exists some generator $g$ of $\mathbb{Z}/n\mathbb{Z}$ such that by reordering the terms  $$T=(\underbrace{g,\ldots,g}\limits_{u}, \ \ \underbrace{-g,\ldots,-g}\limits_{v},  \ \ (x_1 g), \ldots,(x_{\delta-1}g), \ \ (y_1g),\ldots,(y_{w}g))$$
with the following:

(i) $u\geq v\geq 0$ and $u+v=n-2\delta+1$;

(ii) $w=|T|-n+\delta$;

(iii) $x_i$ and $y_j$ are integers with $-\frac{n}{2}<x_i,y_j\leq \frac{n}{2}$ and $x_i\neq 0$  for every $i\in \{1,\ldots,\delta-1\}$ and $j\in \{1,\ldots,w\}$;

(iv) $\sum\limits_{i=1}^{\delta-1} |x_i|\leq 2\delta-2$, where $|x_i|$ is the absolute value of $x_i$.}

By using Theorem B, Gao \cite{GaoDisenumber} characterized all sequences $T$ of $n$ nonzero terms from $\mathbb{Z}/n\mathbb{Z}$ such that the equality $N(T)=n$ holds, and furthermore, substantially improved the result of Theorem A.
One thing worth remarking is that $N(T)\geq 2^{|T|-n+1}$ (see \cite{Olson2}) for any sequence $T$ over $\mathbb{Z}/n\mathbb{Z}$ (without any assumption on the terms of the sequence $T$), and the bound $2^{|T|-n+1}$ is optimal.
Essentially, Theorem B asserts that if $N(T)$ is not large then $T$ yields a
well-structured subsequence, which is called a (signed) smooth sequence (see Definition \ref{Definition behaving sequence}).
The smooth sequence structure plays an important role  for the investigation of inverse zero-sum problems over finite cyclic groups (see  [\cite{GRuzsa}, Theorem 5.1.8], \cite{Yuan,ZengYuan}) and finite cyclic semigroups (see \cite{Wangstructureincyclic}).

In this paper, we shall generalize Theorem B into the setting of finite cyclic semigroups using the smooth sequence structure. The main result will be presented as Theorem \ref{Theorem Number of idempotent-sum subsequences} in Section 3.

\section{Notation and terminology}

For integers $a,b\in \mathbb{Z}$, we set $[a,b]=\{x\in \mathbb{Z}: a\leq x\leq b\}$.
For a real number $x$, we denote by $\lfloor x\rfloor$ the largest integer that is less
than or equal to $x$, and by $\lceil x\rceil$ the smallest integer that is greater than or equal to $x$.

Let $\mathcal{S}$ be a commutative semigroup written additively, where the operation is denoted as $+$.
For any positive integer $m$ and any element $a\in \mathcal{S}$, we denote by $ma$ the sum $\underbrace{a+\cdots+a}\limits_{m}$. An element $e$ of $\mathcal{S}$ is said to be idempotent if $e+ e=e$. Let $E(\mathcal{S})$ be the set consisting of all idempotents of $\mathcal{S}$.
A cyclic semigroup is a semigroup generated by a single element $s$, denoted $\langle s\rangle$, consisting all elements which can be represented as $m s$ for some positive integer $m$.
If the cyclic semigroup $\langle s\rangle$ is infinite then $\langle s\rangle$ is isomorphic to the semigroup of $\mathbb{N}$ with addition (see \cite{Grilletmonograph}, Proposition 5.8), and if $\langle s\rangle$ is finite
then the least integer $k>0$ such that $ks=ts$ for some positive integer $t\neq k$ is called the {\bf index} of $\langle s\rangle$,  then the least integer $n>0$ such that $(k+n)s=k s$ is called the {\bf period} of $\langle s\rangle$. We denote a finite cyclic semigroup of index $k$ and period $n$ by $C_{k; n}$. In particular, if $k=1$ the semigroup $C_{k; n}$ reduces to a cyclic group of order $n$.
For any element $a$ of $C_{k; n}=\langle s\rangle$, let ${\rm Ind}_{s}(a)$ (write as ${\rm Ind}(a)$ for simplicity) be the least positive integer $t$ such that $t s=a$. Let $\mathbb{Z}\diagup n \mathbb{Z}$ be the additive group of integers modulo $n$.
Define a map
\begin{equation}\label{equation map psi}
\psi:C_{k; n}\rightarrow \mathbb{Z}\diagup n \mathbb{Z} \ \ \ \mbox{ given by } \ \
\psi: a\mapsto {\rm Ind}(a)+ n\mathbb{Z} \ \ \ \ \mbox{ for all } a\in C_{k; n}.
\end{equation}

For technical convenience, we  need to introduce notation and terminologies on sequences over semigroups and follow the notation of A. Geroldinger, D.J. Grynkiewicz and
others used for sequences over groups (cf. [\cite{Grynkiewiczmono}, Chapter 10] or [\cite{GH}, Chapter 5]). Let ${\cal F}(\mathcal{S})$
be the free commutative monoid, multiplicatively written, with basis
$\mathcal{S}$. We denote multiplication in $\mathcal{F}(\mathcal{S})$ by the boldsymbol $\cdot$ and we use brackets for all exponentiation in $\mathcal{F}(\mathcal{S})$. By $T\in {\cal F}(\mathcal{S})$, we mean $T$ is a sequence of terms from $\mathcal{S}$ which is
unordered, repetition of terms allowed. Say
$T=a_1a_2\cdot\ldots\cdot a_{\ell}$ where $a_i\in \mathcal{S}$ for $i\in [1,\ell]$.
The sequence $T$ can be also denoted as $T=\mathop{\bullet}\limits_{a\in \mathcal{S}}a^{[{\rm v}_a(T)]},$ where ${\rm v}_a(T)$ is a nonnegative integer and
means that the element $a$ occurs ${\rm v}_a(T)$ times in the sequence $T$. By $|T|$ we denote the length of the sequence, i.e., $|T|=\sum\limits_{a\in \mathcal{S}}{\rm v}_a(T)=\ell.$ By $\varepsilon$ we denote the
{\sl empty sequence} over $\mathcal{S}$ with $|\varepsilon|=0$. We call $T'$
a subsequence of $T$ if ${\rm v}_a(T')\leq {\rm v}_a(T)\ \ \mbox{for each element}\ \ a\in \mathcal{S},$ denoted by $T'\mid T,$ moreover, we write $T^{''}=T\cdot  T'^{[-1]}$ to mean the unique subsequence of $T$ with $T'\cdot T^{''}=T$.  We call $T'$ a {\sl proper} subsequence of $T$ provided that $T'\mid T$ and $T'\neq T$. In particular, the empty sequence  $\varepsilon$ is a proper subsequence of every nonempty sequence. Let $I$ and $J$ be two subsets of $[1,\ell]$, and let
$T_1=\mathop{\bullet}\limits_{i\in I} a_i$ and $T_2=\mathop{\bullet}\limits_{j\in J} a_j$ be two subsequences of $T$. Then we define the intersection of $T_1$ and $T_2$, denoted $T_1\cap T_2$, to be a subsequence of $T$, given as $$T_1\cap T_2=\mathop{\bullet}\limits_{k\in I\cap J} a_k.$$
If $I\neq J$, we say $T_1$ and $T_2$ are {\bf distinct subsequences} of $T$.
Let $\sigma(T)=a_1+\cdots+ a_{\ell}$ be the sum of all terms from $T$.
Let $\Sigma(T)$ be the set consisting of all the elements of $\mathcal{S}$ which can be represented as a sum of one or more terms from $T$, i.e., $\Sigma(T)=\{\sigma(T'): T' \mbox{ is taken over all nonempty subsequences of }T\}$.
We call $T$ a {\bf zero-sum} sequence provided that $\mathcal{S}$ has an identity $0_{\mathcal{S}}$ and $\sigma(T)=0_{\mathcal{S}}$.
In particular,
if $\mathcal{S}$ has an identity $0_{\mathcal{S}}$,  we adopt the convention
that $\sigma(\varepsilon)=0_\mathcal{S}.$
We say  the sequence $T\in \mathcal{F}(\mathcal{S})$ is
\begin{itemize}
     \item a {\bf zero-sum free sequence} if $T$ contains no nonempty zero-sum subsequence;
     \item an {\bf idempotent-sum sequence} if $\sigma(T)$ is an idempotent;
     \item an {\bf idempotent-sum free sequence} if $T$ contains no nonempty idempotent-sum subsequence.
\end{itemize}
It is worth remarking that when the commutative semigroup $\mathcal{S}$ is an abelian group, the notion of {\sl zero-sum free sequence} and {\sl idempotent-sum free sequence} make no difference.

For any nonempty subset $X\subset \mathcal{S}$,
by $N(T; X)$ we denote the number of subsequences of $T$ with sum belonging to the set $X$,
i.e.,
\begin{equation}\label{equationgenereralcounting}
N(T; X)=\sharp\{T': T'\mid T \mbox{ with }\sigma(T')\in X\}.
\end{equation}
In particular, if $X=\{0_{\mathcal{S}}\}$ we shall write $N(T; X)$ as $N(T)$ for short.

For a finite cyclic semigroup $C_{k; n}$, we extend $\psi$ in \eqref{equation map psi} to the map
\begin{equation}\label{equation big pasi}
\Psi:\mathcal{F}(C_{k; n})\rightarrow \mathcal{F}(\mathbb{Z}\diagup n \mathbb{Z}) \ \ \ \mbox{ given by } \ \  \Psi: T\mapsto \mathop{\bullet}\limits_{a\mid T} \psi(a)\ \ \ \ \mbox{ for any sequence }T\in \mathcal{F}(C_{k; n}).
\end{equation}

Now we give the definition of smooth sequences which is crucial to  present the main theorem.

\begin{definition}\label{Definition behaving sequence}  \
Let $G$ be an abelian group, and let $T=a_1\cdot \ldots \cdot a_{\ell}\in \mathcal{F}(G)$ with $\ell=|T|\in \mathbb{N}$.

1. (\cite{GRuzsa}, Definition 5.1.3) \ If $T=(n_1 g)\cdot \ldots\cdot (n_{\ell} g)$, where $g\in G$, $1=n_1\leq \cdots \leq n_{\ell}$, $n=n_1+\cdots +n_{\ell}<{\rm ord}(g)$ and $\sum(T)=\{g,2g,\ldots,ng\}$, we call $T$ a smooth sequence, precisely, a $g$-smooth sequence in this case;

2. If there exist $\epsilon_1,\ldots,\epsilon_{\ell}\in \{1,-1\}$ such that $\mathop{\bullet}\limits_{i\in [1,\ell]} (\epsilon_i \ a_i)\in \mathcal{F}(G)$ is smooth ($g$-smooth), we call $T$ a signed smooth ($g$-smooth) sequence.
\end{definition}

By the very definition, we see that a $g$-smooth sequence is always a signed $g$-smooth sequence, but the converse is not necessarily true (for example, the sequence $T=g\cdot (-g)$ is a signed smooth but not a smooth sequence over the cyclic group $\langle g\rangle$ of order $n>3$).  We also note that for any nonempty sequence $T\in \mathcal{F}(\mathbb{Z})$, $T$ is $1$-smooth if and only if $\sum(T)=[1,\sum\limits_{a\mid T} a]$.

By definition, it is easy to observe the following:

\noindent \textbf{Observation A.} \ {\sl Let $T=(n_1 g)\cdot \ldots\cdot (n_{\ell} g)$ be a sequence over an abelian group $G$, where $g\in G$ and $0<n_1\leq \cdots \leq n_{\ell}\leq {\rm ord}(g)$. Then $T$ is $g$-smooth if and only if the following three conditions holds:

(i) $n_1=1$;

(ii) $n_1+\cdots +n_{\ell}<{\rm ord}(g)$;

(iii) $n_t\leq 1+\sum\limits_{i=1}^{t-1} n_i$ for each $t\in [2,\ell]$.}

\section{The main result}

\begin{lemma}\label{Lemma cyclic semigroup} (\cite{Grilletmonograph},  Chapter I) \ Let $\mathcal{S}=C_{k; n}$ be a finite cyclic semigroup generated by the element $s$. Then  $\mathcal{S}=\{s,\ldots,k s,(k+1)s,\ldots,(k+n-1)s\}$
with
$$\begin{array}{llll} & is+js=\left \{\begin{array}{llll}
               (i+j)s, & \mbox{ if } \  i+j \leq  k+n-1;\\
                ts, &  \mbox{ if }  \ i+j \geq k+n, \ \mbox{ where}  \  k\leq t\leq k+n-1 \ \mbox{ and } \ t\equiv i+j\pmod{n}. \\
              \end{array}
           \right. \\
\end{array}$$
Moreover, there exists a unique idempotent, say $\ell s$, in the cyclic semigroup $\langle s\rangle$, where $$\ell\in [k,k+n-1] \  \mbox{ and }\  \ell\equiv 0\pmod {n}.$$
\end{lemma}

\noindent $\bullet$ In what follows, we shall denote $\textbf{e}$ as the unique idempotent of the cyclic semigroup $C_{k; n}$, and write $N(T; \textbf{e})$ for $N(T; \{\textbf{e}\})$ (given as \eqref{equationgenereralcounting}) for any sequence $T\in \mathcal{F}(C_{k; n})$.

By Lemma \ref{Lemma cyclic semigroup}, it is easy to derive the following.

\begin{lemma}\label{Lemma product condition containing idmepotent} \  Let $k,n>0$, and let $W\in \mathcal{F}(C_{k; n})$ be a nonempty sequence. Then $W$ is an idempotent-sum sequence if, and only if,
$\sum\limits_{a\mid W}{\rm Ind}(a)\geq \left\lceil\frac{k}{n}\right\rceil n$ and $\sum\limits_{a\mid W}{\rm Ind}(a)\equiv 0\pmod{n}$.
\end{lemma}

\begin{lemma}\label{Lemma Olson} \cite{Olson2} \ Let $n>0$, and let $G$ be a finite cyclic group of order $n$.  Then $N(T)\geq 2^{|T|-n+1}$ for any $T\in \mathcal{F}(G)$.
\end{lemma}

\begin{lemma} \label{Lemma IntegersAddition}
For any integers $t,n>0$ and any sequence $T$ of positive integers, there are at least $2^{|T|-tn+1}-1$ distinct nonempty subsequences $W\mid T$ such that $\sigma(W)\geq tn$ and $\sigma(W)\equiv 0\pmod{n}$.
\end{lemma}

\begin{proof} Consider the sequence $\mathop{\bullet}\limits_{a\mid T} (a+m \mathbb{Z})$ over the cyclic group $\mathbb{Z}\diagup m \mathbb{Z}$, where $m=tn$. It follows from Lemma \ref{Lemma Olson} that $N(\mathop{\bullet}\limits_{a\mid T} (a+m \mathbb{Z}))\geq 2^{|\mathop{\bullet}\limits_{a\mid T} (a+m \mathbb{Z})|-m+1}=2^{|T|-m+1}$. That is, $T$ contains at least $2^{|T|-m+1}-1$ distinct nonempty subsequences $W$ such that $\sigma(W)\equiv 0\pmod {m}$. It follows that  $\sigma(W)\equiv 0\pmod {n}$, and that $\sigma(W)\geq m=tn$ because $W$ is nonempty, we are done.
 \end{proof}

The following lemma follows from a straightforward calculation. Below we give a short combinatorial argument for it.

\begin{lemma}\label{Lemma CombinatoricEquation}  For any integers  $m\geq 0$ and $h,n>0$, we have $\sum\limits_{j\geq h} {m\choose j n}\geq 2^{m-hn+1}-1$.
\end{lemma}

\begin{proof} If $m=0$, the conclusion is trivial. Then we assume $m>0$.
Take a sequence $T=1^{[m]}\in \mathcal{F}({\mathbb{Z}})$. We see that $T$ contains exactly $\sum\limits_{j\geq h}{m\choose j n}$ distinct nonempty subsequences $W$ with $\sigma(W)\geq hn$ and $\sigma(W)\equiv 0\pmod n$. Then the conclusion follows from Lemma \ref{Lemma IntegersAddition} immediately.
\end{proof}

\begin{lemma} \label{Lemma last lemma}\  Let $H, H_1,H_2$ be nonempty sequences of positive integers with $|H|\geq 2$.

(i) If $H_1$ is $1$-smooth and $\sigma(H_2)\leq \sigma(H_1)+1$, then $H_1\cdot H_2$ is $1$-smooth;

(ii) If both $H_1$ and $H_2$ are $1$-smooth then $H_1\cdot H_2$ is $1$-smooth;

(iii) If $h$ is the least term of $H$ and $H\cdot h^{[-1]}$ is $1$-smooth, then $H$ is $1$-smooth.
\end{lemma}

\begin{proof} (i) follows from the definition of $1$-smooth sequences. (ii) and (iii) follow from (i).
\end{proof}

\begin{lemma}\cite{Wangstructureincyclic} \label{Lemma behaving for integers} \ For $\ell\geq 1$, let $T$ be a sequence of positive integers  with $|T|=\ell$ such that $T$ is not $1$-smooth. Then $\sigma(T)\geq 2\ell$. Moreover, the equality $\sigma(T)= 2\ell$ holds if and only if either $T=1^{[\ell-1]}\cdot (\ell+1)$ or $T=2^{[\ell]}.$
\end{lemma}

  \begin{lemma} \label{SachenChen} (see \cite{GRuzsa}, Theorem 5.1.8 and Corollary 5.1.10) \  Let $G$ be a finite cyclic group of order $n\geq 3$. If $T\in \mathcal{F}(G)$ is zero-sum free of length at least $\lfloor\frac{n}{2}\rfloor+1$, then there exists some $g\in G$ with ${\rm ord}(g)=n$ such that the following two conclusions hold:

  (i) $T$ is $g$-smooth;

  (ii) ${\rm v}_g(T)\geq 2|T|-n+1$.
 \end{lemma}

\begin{lemma}\label{Lemma instant} \ Let $n>0$ be an integer, and let $T$ be a sequence of positive integers of length $|T|\geq n$ such that $T$ is not $1$-smooth.  Then there exists a nonempty subsequence $K$ of $T$ such that $\sigma(K)\equiv 0\pmod n$ and $\sigma(K) \geq 2|K|
\geq 2(|T|-(n-1))$.
\end{lemma}

\begin{proof} Let $K$ be
a subsequence of $T$ such that
\begin{equation}\label{equationsumainKcongruentpmodn1}
\sigma(K)\equiv 0\pmod n
\end{equation}
with $\sigma(K)$ being maximal.
Then
\begin{equation}\label{equation |T|-|U| less than n}
|T\cdot K^{[-1]}|\leq {\rm D}(\mathbb{Z}\diagup n \mathbb{Z})-1=n-1,
\end{equation}
and $|K|=|T|-|T\cdot K^{[-1]}|\geq n-(n-1)=1$, which implies that
$\sigma(K) \geq n$  by \eqref{equationsumainKcongruentpmodn1}.

Suppose that $K$ is $1$-smooth. Since $T$ is not 1-smooth, it follows from Lemma \ref{Lemma last lemma} (i) that there exists at least one term $b\mid T\cdot K^{[-1]}$ such that $b>\sigma(K)+1$. Since $K$ is $1$-smooth, i.e., $\sum(K)=[1,\sigma(K)]\supseteq [1,n]$, we can take a subsequence $L\mid K$ such that $\sigma(L)\equiv -b\pmod n$. It follows that $L\cdot b$ is a subsequence of $T$ such that $\sigma(L\cdot b)\equiv 0\pmod n$ and $\sigma(L\cdot b)\geq b>\sigma(K)$, a contradiction with the maximality of $\sigma(K)$. Hence, $K$ is not 1-smooth.

Then it follows from \eqref{equation |T|-|U| less than n} and Lemma \ref{Lemma behaving for integers} that
$\sigma(K)
\geq  2|K|=2(|T|-|T\cdot K^{[-1]}|)\geq  2(|T|-(n-1))$, completing the proof.
\end{proof}

To proceed, we need the following definition.

\begin{definition}  Let $G$ be an abelian group, and let $T\in \mathcal{F}(G)$ be a nonempty sequence.
For any $1\leq r\leq |T|$, let $\mathcal{W}_r(T)$ be the set of all signed subsequences of $T$ with length $r$, i.e.,
$$\mathcal{W}_r(T)=\{\mathop{\bullet}\limits_{a\mid V} (\epsilon_a \ a) : \epsilon_a\in \{1,-1\}, V\mid T \mbox{ with } |V|=r\}.$$
\end{definition}

\begin{lemma} \label{Lemma containsasmootheshortsequence} Let $G$ be an abelian group,
and let $T\in \mathcal{F}(G)$ be a nonempty sequence. For any $r\in [1,|T|-1]$, if $\min\{N(L_{r}): L_{r}\in  \mathcal{W}_{r}(T)\}>1$ then
$$\min\{N(L_{r+1}): L_{r+1}\in  \mathcal{W}_{r+1}(T)\}\geq 2 \min\{N(L_{r}): L_{r}\in  \mathcal{W}_{r}(T)\}.$$
\end{lemma}

\begin{proof} \ Let $L\in \mathcal{W}_{r+1}(T)$ with
\begin{equation}\label{equation N0(L)reachminimal}
N(L)=\min\{N(L_{r+1}): L_{r+1}\in  \mathcal{W}_{r+1}(T)\}.
\end{equation}
Note that
\begin{equation}\label{equation minNr+1geqminNr}
\min\{N(L_{r+1}): L_{r+1}\in  \mathcal{W}_{r+1}(T)\}\geq  \min\{N(L_{r}): L_{r}\in  \mathcal{W}_{r}(T)\}.
\end{equation}
 Since $\min\{N(L_{r}): L_{r}\in  \mathcal{W}_{r}(T)\}>1$, it follows from $\eqref{equation minNr+1geqminNr}$ that $N(L)>1$ and so
 $L$ contains a nonempty zero-sum subsequence $V$. Fix a term $b\mid V$.
Let
$U$ be an arbitrary subsequence of $-(V\cdot b^{[-1]})$ (note that $U$ is allowed to be the empty sequence).
We define
\begin{equation}\label{equation tildeU}
\widetilde{U}=b\cdot \left(-\left( [-(V\cdot b^{[-1]})]\cdot U^{[-1]} \right)\right).
\end{equation}
Note that
\begin{equation}\label{equation tildeUdividesV}
\widetilde{U}\mid V.
\end{equation}
Since $\sigma(V)=0_G$, it follows that
\begin{equation}\label{equation tildeU=U}
\begin{array}{llll}
\sigma(\widetilde{U})&=& b+\sigma\left(-\left( [-(V\cdot b^{[-1]})]\cdot U^{[-1]} \right)\right) \\
&=& b-\sigma\left([-(V\cdot b^{[-1]})]\cdot U^{[-1]}\right)\\
&=& b-\left(\sigma(-(V\cdot b^{[-1]}))-\sigma(U)\right)\\
&=& b-\sigma(-(V\cdot b^{[-1]}))+\sigma(U)\\
&=& b+\sigma(V\cdot b^{[-1]})+\sigma(U)\\
&=& \sigma(V)+\sigma(U)\\
&=&\sigma(U). \\
\end{array}
\end{equation}
Let
\begin{equation}\label{equation X1}
\mathcal{X}_1=\left\{A: A \mbox{ is a zero-sum subsequence of } (L\cdot V^{[-1]})\cdot (-(V\cdot b^{[-1]}))\right\}
\end{equation}
 and
\begin{equation}\label{equation X2}
 \mathcal{X}_2=\left\{B: B \mbox{ is a zero-sum subsequence of } L \mbox{ with }b\mid B \right\}.
 \end{equation}
By \eqref{equation tildeU}, \eqref{equation tildeUdividesV} and \eqref{equation tildeU=U}, we can check that the following $\theta: \mathcal{X}_1\rightarrow \mathcal{X}_2$ is an injection of $\mathcal{X}_1$ into $\mathcal{X}_2$, (noting that if $U_1, U_2$ are distinct subsequences of $-(V\cdot b^{[-1]})$ then  $\widetilde{U_1}, \widetilde{U_2}$ are distinct subsequences of $V$).
 For any $A\in \mathcal{X}_1$, let $\theta: A \mapsto A_1\cdot \widetilde{A_2}$
where $A=A_1\cdot A_2$ with $A_1=A\cap (L\cdot V^{[-1]})$
and $A_2=A\cap (-(V\cdot b^{[-1]}))$.

Since $(L\cdot V^{[-1]})\cdot (-(V\cdot b^{[-1]}))\in \mathcal{W}_{r}(T)$ and $L\cdot b^{[-1]}\in \mathcal{W}_{r}(T)$ and $\theta$ is injective, it follows from \eqref{equation N0(L)reachminimal}, \eqref{equation X1} and \eqref{equation X2} that
\begin{equation}
\begin{array}{llll}
\min\{N(L_{r+1}): L_{r+1}\in  \mathcal{W}_{r+1}(T)\}&=& N(L) \\
&=& |\mathcal{X}_2|+N(L\cdot b^{[-1]})\\
&\geq & |\mathcal{X}_1|+N(L\cdot b^{[-1]})\\
&=& N\left(L\cdot V^{[-1]})\cdot (-(V\cdot b^{[-1]})\right)+N(L\cdot b^{[-1]})\\
&\geq & 2 \min\{N(L_{r}): L_{r}\in  \mathcal{W}_{r}(T)\},\\
\end{array}
\end{equation}
completing the proof.
\end{proof}

\begin{lemma} \label{lemma case of k leq n} \ Let $n\geq k\geq 1$, and let $T\in \mathcal{F}(C_{k; n})$ be a nonempty sequence. Then $T$ is an idempotent-sum free [resp. idempotent-sum] sequence if and only if $\Psi(T)=\mathop{\bullet}\limits_{a\mid T} ({\rm Ind}(a)+n \mathbb{Z})\in \mathcal{F}(\mathbb{Z}\diagup n \mathbb{Z})$ is a zero-sum free [resp. zero-sum] sequence. In particular,

\begin{displaymath}
N(T;\textbf{e})=\left\{ \begin{array}{ll}
N(\Psi(T))-1 & \textrm{if $k>1$;}\\
N(\Psi(T)) & \textrm{if $k=1$.}\\
\end{array} \right.
\end{displaymath}
\end{lemma}

 \begin{proof} \ Since $k\leq n$, we see that $\sum\limits_{a\mid W}  {\rm Ind}(a)\equiv 0\pmod n$ implies $\sum\limits_{a\mid W}  {\rm Ind}(a)\geq n=\lceil \frac{k}{n}\rceil n$ for any nonempty sequence $W\in \mathcal{F}(\mathcal{S})$. By Lemma \ref{Lemma product condition containing idmepotent} and the definition of the map $\Psi$, we see that $T$ is an idempotent-sum free [resp. idempotent-sum] sequence if and only if $\Psi(T)\in \mathcal{F}(\mathbb{Z}\diagup n \mathbb{Z})$ is a zero-sum free [resp. zero-sum] sequence, and moreover, we see that
 $\Psi$ is a bijection of the set $$\mathcal{A}=\{ \mbox{nonempty idempotent-sum subsequences of } T\}$$ onto the set $$\mathcal{B}=\{ \mbox{nonempty zero-sum subsequences of } \Psi(T)\},$$ where $\Psi: V\mapsto \Psi(V)$ for any nonempty idempotent-sum subsequence $V$ of $T$  given as \eqref{equation big pasi}.

Suppose $k=1$. Then $C_{k; n}\cong \mathbb{Z}\diagup n \mathbb{Z}$ and the idempotent $\textbf{e}$ is the identity element of $C_{k; n}$. It follows that the empty sequence $\varepsilon$ is both an idempotent-sum subsequence of $T$ and a zero-sum subsequence of $\Psi(T)$.  Then
 $N(T;\textbf{e})=|\mathcal{A}\cup \{\varepsilon\}|=|\mathcal{B}\cup \{\varepsilon\}|=N(\Psi(T))$, done.

Otherwise, $k>1$.  We see that the empty sequence $\varepsilon$ is not an idempotent-sum subsequence of $T$, but an zero-sum subsequence of $\Psi(T)$. Hence, $N(T;\textbf{e})=|\mathcal{A}|=|\mathcal{B}\cup \{\varepsilon\}|-1=N(\Psi(T))-1$, completing the proof of the lemma.
\end{proof}

\begin{theorem}\label{Theorem Number of idempotent-sum subsequences} \
Let $k,n$ be positive integers, and let $T\in \mathcal{F}({\rm C}_{k;  n})$ be a nonempty sequence.  Then $N(T;\textbf{e})\geq 2^{|T|-\left\lceil\frac{k}{n}\right\rceil n+1}-1+\lfloor\frac{1}{k}\rfloor$.  Furthermore, for any integer $\delta>0$ such that $N(T;\textbf{e})<2^{|T|-\lceil \frac{k}{n}\rceil n+1+\delta}-1+\lfloor\frac{1}{k}\rfloor$, then the following hold:

\noindent (i) If $k>n$, then $\mathop{\bullet}\limits_{a\mid T'} {\rm Ind}(a)\in \mathcal{F}(\mathbb{Z})$ is $1$-smooth for any $T'\mid T$ with $|T'|\geq n+\delta-1$;

\noindent (ii) Suppose $k\leq n$ and $\delta\in [1,\lceil\frac{n}{2}\rceil-1]$. Then $T$ contains a subsequence $T'$  of length $n-\delta$
such that
$\mathop{\bullet}\limits_{a\mid T'} ({\rm Ind}(a)+n\mathbb{Z}) \in \mathcal{F}(\mathbb{Z}\diagup n \mathbb{Z})$ is a signed $g$-smooth sequence for some generator $g$ of $\mathbb{Z}\diagup n \mathbb{Z}$.
\end{theorem}

\begin{proof}  By Lemma \ref{Lemma IntegersAddition}, we have that $T$ contains at least $2^{|T|-\left\lceil\frac{k}{n}\right\rceil n+1}-1$ distinct {\sl nonempty} subsequences $W$ such that $\sum\limits_{a\mid W}{\rm Ind}(a)\geq \left\lceil\frac{k}{n}\right\rceil n$ and $\sum\limits_{a\mid W}{\rm Ind}(a)\equiv 0\pmod{n}$. Moreove, if $k=1$, then ${\rm C}_{k;  n}$ reduces to the cyclic group $\mathbb{Z}\diagup n \mathbb{Z}$ and so the empty sequence $\varepsilon$ is an idempotent-sum (zero-sum) subsequence of $T$.  Combined with Lemma \ref{Lemma product condition containing idmepotent}, we have that $N(T;\textbf{e})\geq 2^{|T|-\left\lceil\frac{k}{n}\right\rceil n+1}-1+\lfloor\frac{1}{k}\rfloor$.

\medskip

\noindent {\sl Proof of (i).} \ Note that \begin{equation}\label{equation |T|geq kovernn+n}
|T|\geq \left\lceil\frac{k}{n}\right\rceil n -\delta,
\end{equation}
 for otherwise $N(T;\textbf{e})<2^{|T|-\lceil \frac{k}{n}\rceil n+1+\delta}-1\leq 0$, which is absurd.

\noindent \textbf{Assertion A. } \ The sequence $T$ contains no nonempty subsequence $W$ such that
 $\sum\limits_{a\mid W} {\rm Ind}(a)\geq 2 \delta,$ $\sum\limits_{a\mid W} {\rm Ind}(a)\equiv 0\pmod n$ and $\sum\limits_{a\mid W} {\rm Ind}(a)\geq  2|W|$.

 {\sl Proof of Assertion A.} \ Assume to the contrary that $T$ contains a subsequence $W$ with the given property above.  Let $L_1,\ldots, L_{m}$ be all the distinct nonempty subsequences of $T\cdot W^{[-1]}$ such that
\begin{equation}\label{equation sumLiequivalent0modn}
\sum\limits_{a\mid L_i} {\rm Ind}(a)\equiv 0\pmod n
 \end{equation}
 and
 \begin{equation}\label{equation sumLiequivalent0geqn}
  \sum\limits_{a\mid L_i} {\rm Ind}(a)\geq \left\lceil\frac{k}{n}\right\rceil n- \sum\limits_{a\mid W} {\rm Ind}(a)
   \end{equation}
where $i\in [1,m]$.
Since $\sum\limits_{a\mid W} {\rm Ind}(a)\equiv 0\pmod n$, it follows from
\eqref{equation sumLiequivalent0modn} and \eqref{equation sumLiequivalent0geqn} that $\sum\limits_{a\mid W\cdot L_i} {\rm Ind}(a)\geq \left\lceil \frac{k}{n}\right\rceil n$ and $\sum\limits_{a\mid W\cdot L_i} {\rm Ind}(a)\equiv 0\pmod n$ for each $i\in [1,m]$. Combined with Lemma \ref{Lemma product condition containing idmepotent}, we have that $(W\cdot L_1), \ldots, (W\cdot L_m)$ are $m$ distinct nonempty idempotent-sum subsequences of $T$. Since $\sum\limits_{a\mid W} {\rm Ind}(a)\geq 2 \delta$ and $\sum\limits_{a\mid W} {\rm Ind}(a)\geq  2|W|$, it follows from
Lemma \ref{Lemma product condition containing idmepotent} and Lemma \ref{Lemma IntegersAddition} that
$$\begin{array}{llll}
 N(T; \textbf{e}) &\geq& m \\
 &\geq& -1+2^{|T\cdot W^{[-1]}|-\left(\left\lceil\frac{k}{n}\right\rceil n- \sum\limits_{a\mid W} {\rm Ind}(a) \right)+1}\\
&=& -1+2^{|T|-\left\lceil\frac{k}{n}\right\rceil n+\left(\sum\limits_{a\mid W} {\rm Ind}(a)-|W|\right)+1} \\
&\geq& -1+2^{|T|-\left\lceil\frac{k}{n}\right\rceil n+\frac{\sum\limits_{a\mid W} {\rm Ind}(a)}{2}+1} \\
&\geq& -1+2^{|T|-\left\lceil \frac{k}{n}\right\rceil n+\delta+1}, \\
\end{array}$$
a contradiction. This proves Assertion A. \qed

 To prove conclusion (i), we suppose to the contrary that there exists some subsequence $T'\mid T$ with $|T'|\geq n+\delta-1$ such that $\mathop{\bullet}\limits_{a\mid T'} {\rm Ind}(a)$ is not $1$-smooth.
By Lemma \ref{Lemma instant}, there exists a nonempty subsequence $W\mid T'$ such that $\sum\limits_{a\mid W} {\rm Ind}(a)\equiv 0\pmod n,$ $\sum\limits_{a\mid W} {\rm Ind}(a)\geq  2|W|$ and
$\sum\limits_{a\mid W} {\rm Ind}(a)\geq 2(|T'|-n+1)\geq 2 \delta,$
a contradiction with Assertion A. This proves Conclusion (i).

\noindent {\sl Proof of (ii).} \
 Let
\begin{equation}\label{equation HinWn-delta}
H \in \mathcal{W}_{n-\delta}(\Psi(T))
\end{equation}
 with $$N(H)=\min\left\{N(L_{n-\delta}): L_{n-\delta}\in  \mathcal{W}_{n-\delta}(\Psi(T))\right\}.$$

Suppose that $H\in \mathcal{F}(\mathbb{Z}\diagup n \mathbb{Z})$ is not $g$-smooth for every generator $g$ of $\mathbb{Z}\diagup n \mathbb{Z}$. Since $|H|=n-\delta\geq n-(\lceil\frac{n}{2}\rceil-1)=\lfloor\frac{n}{2}\rfloor+1,$
it follows from Lemma \ref{SachenChen} (i) that $H$ is not zero-sum free (note that if $n=2$ then $|H|\geq  \lfloor\frac{n}{2}\rfloor+1=2={\rm D}(\mathbb{Z}\diagup n \mathbb{Z}$)), and so
$N(H)\geq 2$, equivalently,
\begin{equation}\label{equation minN0Lgeq2}
\min\left\{N(L_{n-\delta}): L_{n-\delta}\in  \mathcal{W}_{n-\delta}(\Psi(T))\right\}\geq 2.
\end{equation}
By \eqref{equation minN0Lgeq2} and applying Lemma \ref{Lemma containsasmootheshortsequence} repeatedly, we derive that
$$\begin{array}{llll}
N(\Psi(T))
&\geq&\min\left\{N(L_{|T|}): L_{|T|}\in  \mathcal{W}_{|T|}(\Psi(T))\right\}\\
&\geq& 2\times \min\left\{N(L_{|T|-1}): L_{|T|-1}\in  \mathcal{W}_{|T|-1}(\Psi(T))\right\}\\
&\vdots&\\
&\geq& 2^{|T|-(n-\delta)} \times \min\left\{N(L_{n-\delta}): L_{n-\delta}\in  \mathcal{W}_{n-\delta}(\Psi(T))\right\} \\
&\geq& 2^{|T|-(n-\delta)+1}. \\
\end{array}$$
It follows from Lemma \ref{lemma case of k leq n} that $N(T; \textbf{e})= N(\Psi(T))-1+\lfloor\frac{1}{k}\rfloor\geq  2^{|T|-(n-\delta)+1}-1+\lfloor\frac{1}{k}\rfloor=2^{|T|-\lceil \frac{k}{n}\rceil n+\delta+1}-1+\lfloor\frac{1}{k}\rfloor$, a contradiction. Hence, $H$ is a $g$-smooth sequence  for some generator $g$ of $\mathbb{Z}\diagup n \mathbb{Z}$.
By \eqref{equation HinWn-delta} and the definition for signed smooth sequences, we have (ii) proved.  This completes the proof of the theorem. \end{proof}

\begin{remark}\label{remark one} It is worth remarking that in Theorem \ref{Theorem Number of idempotent-sum subsequences}, the value $\lceil \frac{k}{n}\rceil n$ appearing in the exponent for the bound of $N(T;\textbf{e})$ is precisely equal to the Erd\H{o}s-Burgess constant of the semigroup ${\rm C}_{k;  n}$, denoted ${\rm I}({\rm C}_{k;  n})$,  which is defined as the least integer $\ell$ such that every sequence $L$ of length $\ell$ over the cyclic semigroup ${\rm C}_{k;  n}$ must contain a nonempty idempotent-sum subsequence.  If $k=1$ the  semigroup ${\rm C}_{k;  n}$ reduces to the cyclic group $\mathbb{Z}/n\mathbb{Z}$, and the Erd\H{o}s-Burgess constant ${\rm I}({\rm C}_{k;  n})=\lceil \frac{k}{n}\rceil n=n$ is precisely the Davenport constant ${\rm D}(\mathbb{Z}/n\mathbb{Z})$ of the cyclic group ${\rm D}(\mathbb{Z}/n\mathbb{Z})$ (see \cite{Wangstructureincyclic,Wangcommu} for example). Hence, the bound $N(T;\textbf{e})\geq 2^{|T|-\left\lceil\frac{k}{n}\right\rceil n+1}-1$ can be written as $N(T;\textbf{e})\geq 2^{|T|-{\rm I}({\rm C}_{k;  n})+1}-1$, which coincide with the case of group, $N(T)\geq  2^{|T|-{\rm D}(\mathbb{Z}/n\mathbb{Z})+1}$. For the connection between enumerations of subsequences and the corresponding zero-sum invariant, we shall discuss more in the final section.
\end{remark}

By Observation A, we can show that any $g$-smooth sequence $T$ of length $\ell$ ($\ell\geq 2$) must contain a subsequence of length $\ell$ which is also $g$-smooth (precisely, $(n_1 g)\cdot \ldots\cdot (n_{\ell-1} g)$ is a $g$-smooth of $T$ given as in Observation A), and therefore, any signed $g$-smooth sequence of length $\ell$ ($\ell\geq 2$) must contain a subsequence of length $\ell-1$ which is also signed $g$-smooth.
From this point of view, the result in Theorem \ref{Theorem Number of idempotent-sum subsequences} (ii) is best possible, i.e., the length $n-\delta$ can not be improved. The reason is as follows. For any $\delta\in [1,\lceil\frac{n}{2}\rceil-1]$, taking a sequence $T$ over the cyclic semigroup ${\rm C}_{k;  n}$ such that $\Psi(T)=(1+n\mathbb{Z})^{[n-\delta]}\cdot (0+n\mathbb{Z})^{[|T|-n+\delta]}$, by Lemma \ref{lemma case of k leq n} we can check that $N(T; \textbf{e})=N(\Psi(T))-1+\lfloor\frac{1}{k}\rfloor
=2^{|T|-n+\delta}-1+\lfloor\frac{1}{k}\rfloor
<2^{|T|-\lceil \frac{k}{n}\rceil n+1+\delta}-1+\lfloor\frac{1}{k}\rfloor$, and that $n-\delta$ is the largest length of signed smooth subsequence of $\Psi(T)$.

\begin{remark}\label{Remark} The restriction $\delta\leq \lceil\frac{n}{2}\rceil-1$ can not be improved for general $n$, which can be seen in Example \ref{Example 1}.
\end{remark}

By Theorem \ref{Theorem Number of idempotent-sum subsequences}, we can determine the structure of the sequence $T$ more precisely in case $k>n$ and
$N(T;\textbf{e})$ is not large given as follows.

\begin{prop} \label{prop 1}\
Let $k,n$ be positive integers with $k>n$, and let $T\in \mathcal{F}({\rm C}_{k;  n})$ be a nonempty sequence. For any $\delta\in [1, \left\lceil\frac{\lceil \frac{k}{n}\rceil n}{2}\right\rceil-\frac{3n}{2}]$, if $N(T;\textbf{e})<2^{|T|-\lceil \frac{k}{n}\rceil n+1+\delta}-1$, then $$\mathop{\bullet}\limits_{a\mid T} {\rm Ind}(a)=1^{[|T|-u]}\cdot x_1\cdot \ldots \cdot x_{u}$$ where $x_1,\ldots,x_u\geq 2$ and $u\leq \sum\limits_{i=1}^u (x_i-1)\leq \delta-1.$
\end{prop}

\begin{proof}

By Theorem \ref{Theorem Number of idempotent-sum subsequences} (i) and the definition for $1$-smooth sequences, we have immediately that
\begin{equation}\label{equation Ind(T)=several1xi}
\mathop{\bullet}\limits_{a\mid T} {\rm Ind}(a)=1^{[|T|-u]}\cdot x_1\cdot \ldots \cdot x_{u}
 \end{equation}
 where
 \begin{equation}\label{equation allxigeq2}
 x_1,\ldots,x_u\geq 2
  \end{equation}
  and
 \begin{equation}\label{equation u in [0,?-1]}
 0\leq u\leq n+\delta-2.
 \end{equation}

Let
\begin{equation}\label{equation sumxi=bn+y}
\sum\limits_{i=1}^u x_i=bn +y \mbox{ with } 1\leq y\leq n.
\end{equation}
Note that \begin{equation}\label{equation new}
|T|\geq \left\lceil\frac{k}{n}\right\rceil n -\delta
\end{equation} for the same argument
given as in the proof of \ref{Theorem Number of idempotent-sum subsequences} (i).
Since $\delta\leq \left\lceil\frac{\lceil \frac{k}{n}\rceil n}{2}\right\rceil-\frac{3n}{2}$, it follows from
 \eqref{equation u in [0,?-1]} and \eqref{equation new} that
$$\begin{array}{llll}
|T|-u-(n-y) &\geq& (\left\lceil\frac{k}{n}\right\rceil n -\delta)- (n+\delta-2)-(n-y)\\
 &=& \left\lceil\frac{k}{n}\right\rceil n-2n-2\delta+y+2\\
&\geq& \left\lceil\frac{k}{n}\right\rceil n-2n-2(\left\lceil\frac{\lceil \frac{k}{n}\rceil n}{2}\right\rceil-\frac{3n}{2})+y+2 \\
&=& \left\lceil\frac{k}{n}\right\rceil n-2\left\lceil\frac{\lceil \frac{k}{n}\rceil n}{2}\right\rceil+n+y+2 \\
&\geq& \left\lceil\frac{k}{n}\right\rceil n-(\lceil \frac{k}{n}\rceil n+1)+n+y+2\\
&=& n+y+1>0.\\
\end{array}$$
Fix a subsequence $H\mid T$ with $\mathop{\bullet}\limits_{a\mid H} {\rm Ind}(a)=1^{[n-y]}\cdot x_1\cdot \ldots \cdot x_{u}$. By \eqref{equation Ind(T)=several1xi}, we have that
 \begin{equation}\label{equation whatTH-1looklike}
 \mathop{\bullet}\limits_{a\mid T\cdot  H^{[-1]}} {\rm Ind}(a)=1^{|T|-u-(n-y)}.
 \end{equation}
Since $\sum\limits_{a\mid H} {\rm Ind}(a)=(b+1)n$, it follows from Lemma \ref{Lemma product condition containing idmepotent} that $L\cdot H$ is an idempotent-sum subsequence of $T$ for any $L\mid (T\cdot  H^{[-1]})$ with $\mathop{\bullet}\limits_{a\mid L}{\rm Ind}(a)\equiv 0\pmod n$ and $(b+1)n+\sum\limits_{a\mid L}{\rm Ind}(a)\geq \lceil \frac{k}{n}\rceil n$.  Then it follows from \eqref{equation whatTH-1looklike} that
 \begin{equation}\label{equation generalN(T)geq}
 N(T; \textbf{e})\geq \sum\limits_{\stackrel{t\geq 0 ~{\rm with} }{ t+b+1\geq \left\lceil \frac{k}{n}\right\rceil}}{|T|-u-(n-y) \choose t n}.
 \end{equation}

 Suppose that $b+1\geq \left\lceil \frac{k}{n}\right\rceil$. By \eqref{equation generalN(T)geq} and Lemma \ref{Lemma CombinatoricEquation}, we have that
 $$\begin{array}{llll}
N(T; \textbf{e})&\geq& \sum\limits_{t\geq 0}{|T|-u-(n-y) \choose t n} \\
&=& 1+\sum\limits_{t\geq 1}{|T|-u-(n-y) \choose t n} \\
&\geq &  1+2^{|T|-u-(n-y)-n+1}-1\\
&=&  2^{|T|-u-(n-y)-n+1}.\\
\end{array}$$
Since $N(T; \textbf{e})<2^{|T|-\left\lceil \frac{k}{n}\right\rceil n+1+\delta}-1$, it follows that $|T|-u-(n-y)-n+1<|T|-\left\lceil \frac{k}{n}\right\rceil n+1+\delta$,
which implies that $\left\lceil \frac{k}{n}\right\rceil n-2n+y+1-u\leq \delta$. By \eqref{equation u in [0,?-1]} and \eqref{equation sumxi=bn+y}, we derive that $2\delta\geq (\left\lceil \frac{k}{n}\right\rceil n-2n+y+1-u)+(u-n+2)=\left\lceil \frac{k}{n}\right\rceil n-3n+3+y
\geq\left\lceil \frac{k}{n}\right\rceil n-3n+4$, a contradiction with $\delta\leq \left\lceil\frac{\lceil \frac{k}{n}\rceil n}{2}\right\rceil-\frac{3n}{2}$. Hence,  $$b+1< \left\lceil \frac{k}{n}\right\rceil.$$

By \eqref{equation generalN(T)geq} and Lemma \ref{Lemma CombinatoricEquation},
we have that
$$\begin{array}{llll}
N(T)
&\geq& \sum\limits_{t\geq \left\lceil \frac{k}{n}\right\rceil-(b+1)}{|T|-u-(n-y) \choose  t n} \\
&\geq& 2^{|T|-u-(n-y)-\left(\left\lceil \frac{k}{n}\right\rceil -(b+1)\right)n+1}-1. \\
\end{array}$$
Since $N(T)<2^{|T|-\left\lceil \frac{k}{n}\right\rceil n+1+\delta}-1$, it follows that $$|T|-u-(n-y)-\left(\left\lceil \frac{k}{n}\right\rceil -(b+1)\right)n+1<|T|-\left\lceil \frac{k}{n}\right\rceil n+1+\delta.$$ Combined with \eqref{equation allxigeq2} and \eqref{equation sumxi=bn+y}, we derive that $$u\leq \sum\limits_{i=1}^u (x_i-1)=bn+y-u\leq \delta-1,$$ completing the proof.
\end{proof}

Now we show that Theorem \ref{Theorem Number of idempotent-sum subsequences} implies Theorem B as a corollary.

\noindent {\bf Sketch of proof.} \ Let $T$ and $\delta$ be the sequence given as Theorem B. Since $\mathbb{Z}\diagup n \mathbb{Z}\cong C_{1;n}$, we can view $T$ to be the sequence over the semigroup $C_{1;n}$.
By Lemma \ref{lemma case of k leq n}, we have $N(T;\textbf{e})=N(T)<2^{|T|-n+1+\delta}=2^{|T|-\lceil \frac{k}{n}\rceil n+1+\delta}-1+\lfloor\frac{1}{k}\rfloor$. By Theorem \ref{Theorem Number of idempotent-sum subsequences} (ii), we have that $T$ contains a subsequence $T'$ of length $n-\delta$
such that
$T'$ is a signed $g$-smooth sequence for some generator $g$ of $\mathbb{Z}\diagup n \mathbb{Z}$. Let $$T'=a_1\cdot\ldots\cdot a_{n-\delta}.$$
By the definition of signed smooth sequences, there exist $\epsilon_1,\ldots,\epsilon_{n-\delta}\in \{1,-1\}$ such that $$L=\mathop{\bullet}\limits_{i\in [1, \ n-\delta]} (\epsilon_i \ a_i)$$ is $g$-smooth, and by Lemma \ref{SachenChen} (ii),
we have ${\rm v}_g(L)\geq  2|L|-n+1=n-2\delta+1$. Suppose $\epsilon_i a_i=g$ for each $i\in [1, n-2\delta+1]$, and
$$\epsilon_j a_j=z_j g \mbox{ where } 1\leq z_j\leq n-1 \mbox{ for each } j\in [n-2\delta+2, n-\delta].$$
 Since $(n-2\delta+1)+\sum\limits_{j=n-2\delta+2}^{n-\delta} z_j\leq n-1$, it follows that $\sum\limits_{j=n-2\delta+2}^{n-\delta} z_j\leq n-1-(n-2\delta+1)=2\delta-2\leq \frac{n}{2}$. Then we conclude that $T$ has the form given as Theorem B. \qed

 The condition $\delta\leq \left\lceil\frac{\lceil \frac{k}{n}\rceil n}{2}\right\rceil-\frac{3n}{2}$ in Proposition \ref{prop 1} is given just for the technical reasons, which is not necessarily optimal to ensure the conclusion holds.
 Hence, we propose the following question:

\begin{ques} What is the maximal value of $\delta$ such that the conclusion in Proposition \ref{prop 1} holds?
\end{ques}

\section{The connection of enumeration of idempotent-sum subsequences and smooth subsequences}

It is easy to observe that a sequence $T$ is idempotent-sum free over cyclic semigroups ${\rm C}_{k;  n}$ (zero-sum free over cyclic groups) if and only if $N(T;\textbf{e})=\lfloor\frac{1}{k}\rfloor$. Hence, the structural characterization of sequences such that $N(T;\textbf{e})$ is not large with respect to a parameter $\delta$ may generalize the result on the structure of long idempotent-sum free sequences. In particular, the first author \cite{Wangstructureincyclic} obtained the following Theorem C on long idempotent-sum free sequences over cyclic semigroups. Below we show that Theorem \ref{Theorem Number of idempotent-sum subsequences} implies Conclusion (i) of Theorem C as a corollary.

\noindent  \textbf{Theorem C.} \cite{Wangstructureincyclic} \ {\sl For integers $k,n\geq 1$, let $T\in \mathcal{F}({\rm C}_{k;  n})$ be a sequence of length
\begin{equation}\label{equation length for T}
|T|\geq \left\{ \begin{array}{ll}
\left\lfloor\frac{(\left\lceil\frac{k}{n}\right\rceil+1) n}{2}\right\rfloor, & \textrm{if $k>n$;}\\
\\
\left\lfloor\frac{n}{2}\right\rfloor+1, & \textrm{otherwise.}\\
\end{array} \right.
\end{equation}
Then $T$ is idempotent-sum free if and only if one of the following two conditions holds:

(i) $\mathop{\bullet}\limits_{a\mid T} {\rm ind}(a)\in \mathcal{F}(\mathbb{Z})$ is $1$-smooth with $\sum\limits_{a\mid T} {\rm ind}(a)\leq \left\lceil\frac{k}{n}\right\rceil n-1$ in the case of $k>n$;

(ii) $\mathop{\bullet}\limits_{a\mid T} ({\rm ind}(a)+n\mathbb{Z}) \in \mathcal{F}(\mathbb{Z}\diagup n \mathbb{Z})$ is $g$-smooth for some $g\in  \mathbb{Z}\diagup n \mathbb{Z}$ with ${\rm ord}(g)=n$ in the case of $k\leq n$.}

\noindent {\bf Sketch of proof.} \  Say $k>n$. Obviously, it suffices to show the necessity, since the converse follows easily from Lemma \ref{Lemma product condition containing idmepotent}. Let $T$ be an idempotent-sum free sequence over ${\rm C}_{k;  n}$ of length $|T|\geq \left\lfloor\frac{(\left\lceil\frac{k}{n}\right\rceil+1) n}{2}\right\rfloor$.
Let $\delta=\left\lceil\frac{k}{n}\right\rceil n-|T|$. Since $T$ is idempotent-sum free, it follows that $|T|<{\rm I}({\rm C}_{k;  n})=\left\lceil\frac{k}{n}\right\rceil n$, and so
$\delta>0$. Moreover, we observe that $N(T;\textbf{e})=0<2^1-1+\lfloor\frac{1}{k}\rfloor=2^{|T|-\lceil \frac{k}{n}\rceil n+1+\delta}-1+\lfloor\frac{1}{k}\rfloor$. Since $n+\delta-1=n+(\left\lceil\frac{k}{n}\right\rceil n-|T|)-1\leq n+(\left\lceil\frac{k}{n}\right\rceil n-\left\lfloor\frac{(\left\lceil\frac{k}{n}\right\rceil+1) n}{2}\right\rfloor)-1=\left\lceil\frac{(\left\lceil\frac{k}{n}\right\rceil+1) n}{2}\right\rceil-1\leq \left\lfloor\frac{(\left\lceil\frac{k}{n}\right\rceil+1) n}{2}\right\rfloor\leq |T|$, it follows from Theorem \ref{Theorem Number of idempotent-sum subsequences} (i) that $\mathop{\bullet}\limits_{a\mid T} {\rm ind}(a)\in \mathcal{F}(\mathbb{Z})$ is $1$-smooth and thus, $\sum\limits_{a\mid T} {\rm ind}(a)\leq \left\lceil\frac{k}{n}\right\rceil n-1$ by Lemma \ref{Lemma product condition containing idmepotent}. \qed

If has been noticed in Section 3 that the enumeration of idempotent-sum subsequences over the cyclic semigroup ${\rm C}_{k;  n}$ has a close connection with the largest length of zero-sum free sequences which are not {\sl signed smooth sequences} over the cyclic group $\mathbb{Z}\diagup n \mathbb{Z}$. To learn more on this connection in detail, we need to propose the following definition.

\begin{definition} Let $G$ be a finite cyclic group. We define ${\rm Sgn}(G)$ to be the smallest positive integer $\ell$ such that every zero-sum free sequence $T$ over $G$ of length greater than or equal to $\ell$ is a singed $g$-smooth sequence for some generator $g$ of $G$.
\end{definition}

Then we have the following proposition.

\begin{prop} \label{prop sig} Let $n\geq 2$ be a positive integer, and let $G$ be a cyclic group of order $n$. \begin{displaymath}
{\rm Sgn}(G)=\left\{ \begin{array}{ll}
1 & \textrm{if $n\in \{2,3,5,7\}$;}\\
2 & \textrm{if $n=4$;}\\
\lceil \frac{n+1}{2}\rceil & \textrm{else.}\\
\end{array} \right.
\end{displaymath}
\end{prop}

\begin{proof}  If $n\in \{2,3,4,5,7\}$ then the conclusion follows by straightforward verifications. Hence, we assume $n\notin \{2, 3,4,5,7\}$. Since $a$ g-smooth sequence is also a signed $g$-smooth sequence, it follows from Lemma \ref{SachenChen} (i) that
$${\rm Sgn}(G)\leq \left\lfloor\frac{n}{2}\right\rfloor+1=\left\lceil \frac{n+1}{2}\right\rceil$$ for any integer $n\geq 3$. To prove ${\rm Sgn}(G)=\lceil \frac{n+1}{2}\rceil$, it suffices to construct a zero-sum free sequence $L$ of length $\lceil \frac{n+1}{2}\rceil-1$ over $G$ such that $L$ is not a signed $g$-smooth sequence for every generator $g$ of $G$.

Now we fix a generator $x$ of $G$. We can check that the following sequence $L$ is the desired one, where \begin{displaymath}
L=\left\{ \begin{array}{ll}
(2x)^{[\frac{n-5}{2}]}\cdot(3 x)^{[2]} & \textrm{if $n\equiv 1 \pmod 2$  and  $n\geq 11$;}\\
(2x)^{[\frac{n}{2}-1]}\cdot(3 x) & \textrm{if $n\equiv 0 \pmod 2$  and  $n\geq 6$};\\
x\cdot(3x)^{[2]}\cdot (7x) & \textrm{if $n=9$}.\\
\end{array} \right.
\end{displaymath} \end{proof}

In fact,  an additive invariant, denoted $\widehat{{\rm Smo}}(\cdot)$, similar to ${\rm Sgn}(G)$ has been investigated for cyclic semigroups, which also originates from the research done by Chapman, Freeze and Smith \cite{Chapman,Chapman2}, Gao \cite{Gaointeger}, Savchev and Chen \cite{Sav}, and Yuan
\cite{Yuan} for minimal zero-sum (zero-sum free sequences) sequences over finite cyclic groups.  For convenience, we introduce only its form in cyclic groups as follows.

\begin{definition}  \cite{Wangstructureincyclic} Let $G$ be a finite cyclic group. We define $\widehat{{\rm Smo}}(G)$ to be the smallest positive integer $\ell$ such that every zero-sum free sequence $T$ over $G$ of length greater than or equal to $\ell$ is a $g$-smooth sequence for some generator $g$ of $G$.
\end{definition}

By Theorem 4.4 in \cite{Wangstructureincyclic}, we have the following.

\begin{prop} Let $n\geq 2$ be a positive integer, and let $G$ be a cyclic group of order $n$. \begin{displaymath}
\widehat{{\rm Smo}}(G)=\left\{ \begin{array}{ll}
1 & \textrm{if $n\in \{2,3,5\}$;}\\
2 & \textrm{if $n=4$;}\\
3 & \textrm{if $n=7$;}\\
\lceil \frac{n+1}{2}\rceil & \textrm{else.}\\
\end{array} \right.
\end{displaymath}
\end{prop}

For more results and open questions on $\widehat{{\rm Smo}}({\rm C}_{k;  n})$ of cyclic semigroups ${\rm C}_{k;  n}$, one is referred to Section 3 in \cite{Wangstructureincyclic}. Below we shall give an example by Proposition \ref{prop sig} to show the restriction $\delta\leq \lceil\frac{n}{2}\rceil-1$ in Theorem \ref{Theorem Number of idempotent-sum subsequences} can not be improved just as stated in Remark \ref{Remark}.

\begin{exam}\label{Example 1}
Suppose $n\equiv 0\pmod 2$ with $n\geq 6$, or $n\equiv 1\pmod 2$ with $n\geq 9$.
Set $\delta=\lceil\frac{n}{2}\rceil$. By Proposition  \ref{prop sig}, we can take a sequence $L$ of length ${\rm Sgn}(\mathbb{Z}\diagup n \mathbb{Z})-1=\lceil \frac{n+1}{2}\rceil-1=\lfloor \frac{n}{2}\rfloor$ over ${\rm C}_{k;  n}$ such that $\Psi(L)$ is a zero-sum free sequence over $\mathbb{Z}\diagup n \mathbb{Z}$, and that $\Psi(L)$ is not a signed $g$-smooth for every generator $g$ of $\mathbb{Z}\diagup n \mathbb{Z}$.
Then we set $T=L\cdot (na)^{[|T|-\lfloor \frac{n}{2}\rfloor]}$, i.e., $\Psi(T)=\Psi(L)\cdot (0+n\mathbb{Z})^{[|T|-\lfloor \frac{n}{2}\rfloor]}$.  Since the sequence $\Psi(L)$ is zero-sum free over $\mathbb{Z}\diagup n \mathbb{Z}$, it follows from Lemma \ref{lemma case of k leq n} that
$N(T; \textbf{e})=N(\Psi(T))-1+\lfloor\frac{1}{k}\rfloor
=2^{|T|-\lfloor \frac{n}{2}\rfloor}-1+\lfloor\frac{1}{k}\rfloor
=2^{|T|-(n-\lceil \frac{n}{2}\rceil)}-1+\lfloor\frac{1}{k}\rfloor
=2^{|T|-\lceil \frac{k}{n}\rceil n+\delta}-1+\lfloor\frac{1}{k}\rfloor
<2^{|T|-\lceil \frac{k}{n}\rceil n+1+\delta}-1+\lfloor\frac{1}{k}\rfloor$. It follows that  $\Psi(T)$ contains no signed $g$-smooth subsequence of length $n-\delta=\lfloor\frac{n}{2}\rfloor$, where $g$ runs over every generator of $\mathbb{Z}\diagup n \mathbb{Z}$.
\end{exam}

As stated in Remark \ref{remark one},  the enumeration of subsequences with a given additive property has a close connection with the corresponding zero-sum invariant, precisely,
the number of subsequences with a given property (idempotent-sum, zero-sum, etc.) is bounded by a power of $2$ with the exponent to be a function of the corresponding zero-sum invariant which ensures the existence of subsequences with the prescribed properties (idempotent-sum or zero-sum).  Hence, we shall close this paper with the following general problem:

\begin{problem} Let $\mathcal{S}$ be a finite commutative semigroup (or group). Determine all subsets $X$ of $\mathcal{S}$ which possess the following two properties simultaneously:

(i) There exists a least positive integer, denoted $Z(X)$, such that every sequence over $\mathcal{S}$ of length $Z(X)$ must contain a nonempty subsequence with sum belonging to $X$?

(ii) There exists a linear function $f(\cdot)$ and a constant $\lambda$, such that for every integer $\ell \geq Z(X)$, the value of $\min\left\{N(T; X): T \mbox{ is taken over all sequences of terms from $\mathcal{S}$ with } |T|=\ell\right\}$ is equal to $2^{f\left(\ell-Z(X)\right)}+\lambda$?
 \end{problem}

 \begin{remark} By \cite{Olson2} Theorem 2 (which can be also found in \cite{ChangQuWang}), we can show that if $\mathcal{S}$ is a finite commutative group and $X$ a subgroup of $\mathcal{S}$, then $X$ possesses the above two properties, in particular, $Z(X)={\rm D}(G\diagup X)$ and $$\min\left\{N(T; X): T \mbox{ is taken over all sequences of terms from $\mathcal{S}$ with } |T|=\ell\right\}=2^{\ell-Z(X)+1}$$ for any integer $\ell \geq Z(X)$.

 \end{remark}

\bigskip

\noindent {\bf Acknowledgements}

\noindent  This work is supported by NSFC (grant no. 12371335,  12271520).

\end{document}